\documentclass[12pt]{article}
\usepackage{epsfig}
\usepackage{amsfonts}
\usepackage{amssymb}
\usepackage{amstext}
\usepackage{amsmath}
\usepackage{xspace}
\usepackage{theorem}
\usepackage{graphicx}
\makeatletter
\newenvironment{proof}{{\bf Proof:  }}{\hfill\rule{2mm}{2mm}}

\newcommand{\junk}[1]{}
\newtheorem{theorem}{Theorem}
\newtheorem{lemma}[theorem]{Lemma}

\newtheorem{Lemma}[theorem]{Lemma}

\newtheorem{corollary}[theorem]{Corollary}
\newtheorem{definition}{Definition}
\newtheorem{remark}{Remark}

\newcommand{\Aut}{\ensuremath{\textnormal{Aut}}}

\newcommand{\x}{\ensuremath{{\bf x}}}
\newcommand{\A}{\ensuremath{\mathcal A}}

\newcommand{\C}{\ensuremath{\mathbb{C}}}

\DeclareMathOperator{\tr}{tr}
\DeclareMathOperator{\Tr}{Tr}
\title{A Harary-Sachs Theorem for Hypergraphs}
\author{Gregory J. Clark and Joshua N. Cooper\\
\small Sa\"id Business School\\[-0.8ex]
\small University of Oxford\\
\small \texttt{gregory.clark@sbs.ox.ac.uk  }\\
\small Department of Mathematics\\[-0.8ex]
\small University of South Carolina\\
\small \texttt{cooper@math.sc.edu}\\
}
\begin{document}
\maketitle

\begin{abstract}
We generalize the Harary-Sachs theorem to $k$-uniform hypergraphs: the codegree-$d$ coefficient of the characteristic polynomial of a uniform hypergraph ${\cal H}$ can be expressed as a weighted sum of subgraph counts over certain multi-hypergraphs with $d$ edges.  We include a detailed description of the aforementioned multi-hypergraphs and a formula for their corresponding weights.
\end{abstract}
\section{Introduction}
An early, seminal result in spectral graph theory of Harary \cite{Har} (and later, more explicitly, Sachs \cite{Sac}) showed how to express the coefficients of a graph's characteristic polynomial as a certain weighted sum of the counts of various subgraphs of $G$ (a thorough treatment of the subject is given in \cite{Big}, Chapter 7).
\begin{theorem}
\label{T:Harary}
(\cite{Har},\cite{Sac}) Let $G$ be a labeled simple graph on $n$ vertices. If $H_i$ denotes the collection of $i$-vertex graphs whose components are edges or cycles, and $c_i$ denotes the coefficient of $\lambda^{n-i}$ in the characteristic polynomial of $G$, then 
\[
c_{i} = \sum_{H \in H_i} (-1)^{c(H)}2^{z(H)} [\# H \subseteq G]
\]
where $c(H)$ is the number of components of $H$, $z(H)$ is the number of components which are cycles, and $[\# H \subseteq G]$ denotes the number of (labeled) subgraphs of $G$ which are isomorphic to $H$.
\end{theorem} 
The goal of the present paper is to provide an analogous result for the characteristic polynomial of a hypergraph.  The full result is given in Theorem \ref{T:codegreeFormula}, but to state here simply: fix $k \geq 2$ and let $H_d$ denote the set of $k$-valent (i.e., $k$ divides the degree of each vertex) $k$-uniform multi-hypergraphs on $d$ edges.  For a $k$-uniform hypergraph $\cal H$ the codegree-$d$ coefficient (i.e., the coefficient of $x^{\deg - d}$) of the characteristic polynomial of the $n$-vertex hypergraph ${\cal H}$ can be written
\[
c_d = \sum_{H \in H_d}(-(k-1)^n)^{c(H)}C_H(\# H \subseteq {\cal H})
\]
where $c(H)$ is the number of components of $H$, $C_H$ is a constant depending only on $H$, and $(\# H \subseteq {\cal H})$ is the number of times $H$ occurs (in a certain sense that is a minor generalization of the subgraph relation) in ${\cal H}$.

The quantity $(\# H \subseteq {\cal H})$ is straightforward to compute.  However, computing $C_H$ is more complicated.  This notion of an \emph{associated coefficient of a hypergraph} first appeared in \cite{Coo}, where the authors provide a combinatorial description of the codegree $k$ and codegree $k+1$ coefficient, denoted $c_k$ and $c_{k+1}$ respectively, for the normalized adjacency characteristic polynomial of a $k$-uniform hypergraph. 
\begin{theorem}
\label{T:Codegree-k}
\cite{Coo} Let ${\cal H}$ be a $k$-uniform hypergraph.  Then
\[
c_k=-k^{k-2}(k-1)^{n-k}|E({\cal H})|
\]
and
\[
c_{k+1} = -C_k (k-1)^{n-k}(\#\text{ of simplices in ${\cal H}$}),
\]
 where $C_k$ is some constant depending on $k$.
\end{theorem}
This idea was further studied by Shao, Qi, and Hu where the authors prove (restating Theorem 4.1 of \cite{Sha}),
\[
c_d = (k-1)^{n-1} \sum_{D \in {\bf D}} f_D|\mathfrak{E}(D)|
\]
where ${\bf D}$ is a certain large family of digraphs, $f_D$ is a function of $D$ and $\mathfrak{E}(D)$ is the set of Euler circuits in $D$.  The authors then use their formula to provide a general description of $\Tr_2(T)$ and $\Tr_3(T)$ for a general tensor $T$.  Our first few results are similar to that of \cite{Sha} (as described in more detail below), and we use them to provide an explicit combinatorial description of $H_D$ and the resulting $C_H$. 

Here we present some requisite background maintaining the notation of \cite{Coo}. A {\em (cubical) hypermatrix} $\A$ over a set $\mathbb{S}$ of {\em dimension} $n$ and {\em order} $k$ is a collection of $n^k$ elements $a_{i_1i_2\dots i_k} \in \mathbb{S}$ where $i_j \in [n]$.   A hypermatrix is said to be {\em symmetric} if entries with identical multisets of indices are the same.  That is, $\A$ is symmetric if $a_{i_1i_2\dots i_k} = a_{i_{\sigma(1)}i_{\sigma(2)}\dots i_{\sigma(k)}}$ for all permutations $\sigma$ of $[k]$.  An order $k$ dimension $n$ symmetric hypermatrix $\A$ uniquely defines a homogeneous degree $k$ polynomial in $n$ variables (a.k.a.~a ``$k$-form'') by 
\[
F_{\A}({\bf x}) = \sum_{i_1, i_2, \dots, i_k =1}^na_{i_1i_2\dots i_k} x_{i_1}x_{i_2}\dots x_{i_k}.
\]
 If we write $\x^{\otimes r}$ for the order $r$ dimension $n$ hypermatrix with $i_1, i_2, \dots, i_k$ entry $x_{1_1}x_{i_2} \dots x_{i_r}$ and $x^r$ for the vector with $i$-th entry $x_i^r$ then the above expression can be written as
\[
\A\x^{\otimes k-1} = \lambda \x^{k-1}
\] where the multiplication denoted by concatenation is tensor contraction. Call $\lambda \in \mathbb{C}$ an {\em eigenvalue} of $\A$ if there is a non-zero vector ${\bf x} \in \mathbb{C}^n$, which we call an {\em eigenvector}, satisfying 
\[
\sum_{ i_2, i_3, \dots, i_k =1}^na_{ji_2\dots i_k} x_{i_1}x_{i_2}\dots x_{i_k} = \lambda x_j^{k-1}.
\]
Next we offer an important result from commutative algebra to proceed the definition of the adjacency characteristic polynomial of a hypergraph. 
\begin{theorem}
\label{t:Resultant}
(The Resultant, \cite{Gel}) Fix degrees $d_1, d_2, \dots, d_n$.  For $i \in [n]$, consider all monomials $\x^{\alpha}$ (where $\alpha$ is itself a vector) of total degree $d_i$ in $x_1, \dots, x_n$.  For each such monomial, define a variable $u_{i, \alpha}$.  Then there is a unique polynomial $\textsc{res} \in \mathbb{Z}[\{u_{i, \alpha}\}]$ with the following three properties:
\begin{enumerate}
\item{If $F_1, \dots, F_n \in \mathbb{C}[x_1, \dots, x_n]$ are homogeneous polynomials of degrees $d_1, \dots, d_n$ respectively, then the polynomials have a non-trivial common root in $\mathbb{C}^n$ exactly when $\textsc{res}(F_1, \dots, F_n)=0$.  Here, $\textsc{res}(F_1, \dots, F_n)$ is interpreted to mean substituting the coefficient of $\x^\alpha$ in $F_i$ for the variable $u_{i,\alpha}$ in $\textsc{res}$.}
\item{$\textsc{res}(x_1^{d_1}, \dots, x_n^{d_n}) = 1$.}
\item{$\textsc{res}$ is irreducible, even in $\mathbb{C}[\{u_{i,\alpha}\}]$.}
\end{enumerate}
Moreover, for $i \in [n]$, $\textsc{res}$ is homogeneous in the variable $\{u_{i,\alpha}\}$ with degree $\prod_{j \in[n], j \neq i} d_i$.
\end{theorem}

\begin{definition} (\cite{Qi})
The symmetric hyperdeterminant of $\A$, denoted $\det(\A)$ is the resultant of the polynomials which comprise the coordinates of $\A x^{\otimes k-1}$.  Let $\lambda$ be an indeterminate.  The characteristic polynomial $\phi_\A(\lambda)$ of a hypermatrix $\A$ is $\phi_\A(\lambda) = \det(\lambda {\cal I} - \A)$.  
\end{definition}
We consider the normalized adjacency matrix of a $k$-uniform hypergraph, ${\cal H} = (V,E)$.  We refer to such hypergraphs as \emph{$k$-graphs} and we reserve the language of \emph{graph} for the case of $k=2$.  For a $k$-graph ${\cal H} = ([n],E)$ we denote the \emph{(normalized) adjacency hypermatrix} ${\cal A}_{\cal H}$ to be the order $k$ dimension $n$ hypermatrix with entries 
\begin{displaymath}
   a_{i_1,i_2, \dots, i_k} = \frac{1}{(k-1)!} \left\{
     \begin{array}{ll}
       1 & : \{i_1, i_2, \dots, i_k\} \in E(H)\\
       0 & : \text{otherwise.}
     \end{array}
   \right.
\end{displaymath} 
For simplicity, we denote $\phi({\cal H}) = \phi_{{\cal A}_{\cal H}}(\lambda)$ and write
\[
\phi({\cal H}) = \sum_{i=0}^{t}c_i \lambda^{t-i}
\]
where $t = n(k-1)^{n-1}$ by Theorem \ref{t:Resultant}.  Throughout, we make use of the notation $\phi_d({\cal H}) = c_d$ for the codegree-$d$ coefficient of $\phi({\cal H})$.  

Our approach relies on the following trace formula for the hyperdeterminant of a tensor.  In \cite{Mor}, Morozov and Shakirov give a formula for calculating $\det({\cal I} - {\cal A})$ using Schur polynomials in the generalized traces of the order $k$, dimension $n$ hypermatrix ${\cal A}$.
Let $f: \C^n \to C^n$ be a linear map and let $I$ be the unity map, $I = (x_1, x_2, \dots, x_n)^T \to (x_1, x_2, \dots, x_n)^T$.  Famously, 
\[
\log \det(I-f) = \tr \log(I - f) = -\sum_{k=1}^\infty \frac{\tr(f^k)}{k}.
\]
The characteristic polynomial is defined as the resultant of a certain system of equations, so calculating the characteristic polynomial requires computation of the resultant.  Moroz and Shakirov give a formula for calculating $\det({\cal I} - \A)$ using \emph{Schur polynomials} in the generalized traces of the order $k$, dimension $n$ hypermatrix $\A$.  
\begin{definition}
Define the $d$-th Schur polynomial $P_d \in \mathbb{Z}[t_1, \dots, t_d]$ by $P_0 = 1$ and, for $d > 0$, 
\[
P_d(t_1, \dots, t_d) = \sum_{m=1}^d \sum_{d_1 + \dots + d_m = d} \frac{t_{d_1}\cdots t_{d_m}}{m!}.
\]
  More compactly, one may define $P_d$ by 
\[
\exp\left(\sum_{d=1}^\infty t_dz^d \right) = \sum_{d=1}^\infty P_d(t_1, \dots, t_d)z^d.
\]
\end{definition}
Let $f_i$ denote the $i$th coordinate of $\A\x^{\otimes k-1}$.  Define $A$ to be an auxiliary $n \times n$ matrix with distinct variables $A_{ij}$ as entries.  For each $I$, we define the differential operator 
\[
\hat f_i = f_i\left(\frac{\partial}{\partial A_{i1}}, \frac{\partial}{\partial A_{i2}}, \dots , \frac{\partial}{\partial A_{in}}\right)
\]
in the natural way.  In \cite{Coo}, Cooper and Dutle use the aforementioned Morozov-Shakirov formula to show that the $d$-th trace of $\cal{A}_H$, 
\begin{equation}
\label{E:Morozov}
\Tr_d({\cal A_{H}}) = (k-1)^{n-1} \sum_{d_1 + \dots +d_n = d} \left(\prod_{i=1}^n \frac{\hat{f_i}^{d_i}}{(d_i(k-1))!} \tr(A^{d(k-1)})\right)
\end{equation}
where $\tr(A^{d(k-1)})$ is the standard matrix trace (for a more detailed explanation, see \cite{Coo}).  We prove Theorem \ref{T:codegreeFormula} with the aid of the following reformulation of Equation \ref{E:Morozov}:
\[
\Tr_d({\cal A_{H}}) = (k-1)^n \sum_{H \in H_d}C_H(\#H \subseteq {\cal H}).
\]

The paper is arranged as follows.  In the following section we define the associated digraph of an operator in $\Tr_d({\cal H})$, as in Equation \ref{E:Morozov}.  In particular, we provide a formula for a summand of $\Tr_d({\cal H})$ in terms of associated digraphs.  In Section 3 we give a combinatorial description of the differential operators which have non-zero contribution to $\Tr_d({\cal H})$ and characterize these operators in terms of associated digraphs.  In Section 4 we show that associated digraphs correspond to a particular type of hypergraph, termed Veblen hypergraphs. We conclude with a proof of our main result. 


\section{The associated digraph of an operator}
Recall the $d$-th trace of $\cal{A}_H$ from Equation \ref{E:Morozov}, 
\[
\Tr_d({\cal A_{H}}) = (k-1)^{n-1} \sum_{d_1 + \dots +d_n = d} \left(\prod_{i=1}^n \frac{\hat{f_i}^{d_i}}{(d_i(k-1))!} \tr(A^{d(k-1)})\right)
\] where $\tr(A^{d(k-1)})$ is the standard matrix operation.  Let $\hat f_{d_1, d_2, \dots, d_n}$ be an addend of $\prod_{i=1}^n \hat f_i^{d_i}$ in $\Tr_d({\cal A_{H}})$.  When the context is clear we suppress the subscript and simply write $\hat f$.  Given $\alpha = (i_1, i_2, \dots, i_{d(k-1)})$ let 
\begin{equation}
\label{E:alpha}
A_\alpha := A_{i_1, i_2}A_{i_2, i_3}\dots A_{i_{d(k-1)-1}, i_{d(k-1)}} A_{i_{d(k-1), i_1}}\end{equation}  and recall that 
\[
\tr(A^{d(k-1)}) = \sum_{\alpha} A_\alpha
\]
where the factors of $A_\alpha$ are commutative. Adhering to the terminology of \cite{Coo} we say $A_\alpha$ is $k$-valent if $k$ divides the number of times $i$ occurs in a subscript of $A_{\alpha}$.   We utilize divisibility notation for monomials in $\tr(A^{d(k-1)})$, e.g., using $g|h$ to denote that $g$ occurs as a factor of the formal product $h$.  We say that \emph{$A_\alpha$ survives $\hat f$} if $\hat f A_{\alpha} \neq 0$.
\begin{definition}
For a differential operator $\hat f_{d_1, d_2, \dots, d_n}$ the associated digraph of $\hat f$, denoted $D_{\hat f}$, is the directed multigraph where there are $d_i$ distinguishable edges directed from $i$ to $j$ given $\left(\frac{\partial}{\partial A_{i,j}}\right)^{d_i} \mid \hat f$ and isolated vertices are ignored. 
\end{definition}
We suppress the subscript and write $D$ when $\hat f$ is understood.
We recall the following graph theoretic definitions according to \cite{Die}. 
\begin{definition}
A \emph{walk} in a graph $G$ is a non-empty alternating sequence $v_0e_0v_1e_1\dots e_{k-1}v_k$ of vertices and edges in $G$ such that $e_i = \{v_i, v_{i+1}\}$ for all $i < k$.  A walk is \emph{closed} if $v_0 = v_k$.  A closed walk in a graph is an \emph{Euler tour} if it traverses every edge of the graph exactly once.  An \emph{Euler circuit} is an Euler tour up to cyclic permutation of its edges, i.e., an Euler tour with no distinguished beginning.  
\end{definition}
We denote the set of Euler tours of a graph $G$ which begin at the edge $e \in E(G)$  by $\mathfrak{E}_{e}(G)$ and we denote the set of Euler circuits of $G$ by $\mathfrak{E}(G)$.   Recall that a  digraph $D$ has an Euler circuit if and only if $\deg^+(v) = \deg^-(v)$ for all $v \in V(D)$ and $D$ is weakly connected. 
\begin{definition}
Let $w = (v_i)_{i=0}^m$ be a sequence of (not necessarily distinct) vertices of $D$.  We say that \emph{$w$ describes an Euler tour in $D$} if there exist distinct edges $e_0, \dots, e_{m-1}$ such that $v_0e_0v_1e_1\dots e_{m-1}v_m$ is an Euler tour in $D$.  Moreover we say that such Euler tours are \emph{described by $w$}.
\end{definition}
Note that the use of Euler tour in the previous definition is well-founded as $e_0$ is distinguished as the first edge. 
\begin{lemma}
\label{L:opForm}
Consider $\Tr_d({\cal H})$, $\hat f \tr(A^{d(k-1)})\neq 0$ if and only if $D$ is Eulerian.  In this case 
\begin{equation}
\label{E:Eulerian}
\hat f \tr(A^{d(k-1)}) = |E(D)||\mathfrak{E}(D)|.
\end{equation}
\end{lemma}
\begin{proof}
Consider $\Tr_d({\cal H})$.  Fix a term $A_\alpha$ of $\tr(A^{d(k-1)})$ and a differential operator $\hat f$ of $\prod_{i=1}^n \hat f_i^{d_i}$.  Suppose $\hat f A_{\alpha} \neq 0$.  Whence $\hat f A_{\alpha} \neq 0$ the factors of $\hat f$ are in one-to-one correspondence with the factors of $A_\alpha$.  It follows that the edges of $D_{\hat f}$ are in one-to-one correspondence with the factors of $A_\alpha$.  Notice that for $i \in V(D)$, $\deg^+(i)  = \deg^-(i)$ by Equation \ref{E:alpha}. Further, $D$ is strongly connected as the sequence of indices from $i$ to $j$ (cyclically, if necessary) is a walk from vertex $i$ to vertex $j$.  Therefore, $D$ is Eulerian. 

Suppose now that $D$ is Eulerian and let $\alpha = (v_i)_{i=0}$ describe an Euler tour in $D$.  We claim that $\hat f A_{\alpha}$ is equal to the number of Euler tours in $D_{\hat f}$ described by $\alpha$.  Let $m(i,j)$ denote the number of edges from $i$ to $j$ in $D$.  Notice that 
\[
A_\alpha = \prod_{i,j \in V(D)} A_{i,j}^{m(i,j)} \text{ and } \hat f = \prod_{i,j \in V(D)} \frac{\partial^{m(i,j)}}{\partial A^{m(i,j)}_{i,j}}.
\]
 Clearly 
\[
\hat f A_\alpha = \prod_{i,j \in V(D)}m(i,j)!.
\]
Moreover, $\alpha$ describes $\prod_{i,j \in V(D)}m(i,j)!$ Euler tours in $D$ by straightforward enumeration.  Observe that there are $|E(D)||\mathfrak{E}(D)|$ Euler tours in $D$ as we may distinguish any edge from $D$ as the first edge of an Euler circuit.  Thus, 
\[
\hat f \tr(A^{d(k-1)}) = \sum_\alpha \hat f A_\alpha = |E(D)||\mathfrak{E}(D)|
\]
 as every Euler tour is described by exactly one $\alpha$.  
\end{proof}

We conclude this section with a remark about the evaluation of Equation \ref{E:Eulerian}.  Conveniently, $|\mathfrak{E}(D)| $ can be computed using the BEST theorem, originally appearing in \cite{Aar} as a variation of a result of \cite{Tut}.
\begin{theorem}
\label{T:BEST}
(BEST Theorem) The number of Euler circuits in a connected Eulerian graph $G$ is 
\[
|\mathfrak{E}(G)| = \tau(G) \prod_{v \in V} (\deg(v) -1)!
\]
 where $\tau(G)$ is the number of arborescences (i.e., the number of rooted subtrees of $G$ with a specified root).
\end{theorem}
For simplicity we abbreviate $\tau(f) = \tau(D_{\hat f})$.  Combining the BEST theorem and the observation that $|E(D)| = d(k-1)$ yields the following.
\begin{corollary}
\label{C:opForm}
\[
\hat f \tr(A^{d(k-1)}) = d(k-1) \tau(f) \prod_{v \in V(D)}(\deg^-(v) -1)!
\]
\end{corollary}
As a final note, recall that $\tau(G)$ can be computed using the Matrix Tree Theorem, which makes the computation of the right-hand side of the equality in Corollary \ref{C:opForm} efficient.
\begin{theorem}
\label{T:Kirchhoff}
(Matrix Tree Theorem/Kirchhoff's Theorem) For a given connected graph $G$ with $n$ labeled vertices, let $\lambda_1, \lambda_2, \dots \lambda_{n-1}$ be the non-zero eigenvalues of ${\cal L}(G) = D(G) - A(G)$.  Then 
\[
\tau(G) = \frac{\lambda_1\lambda_2 \dots \lambda_{n-1}}{n}.
\]
\end{theorem}

\section{Euler operators and Veblen hypergraphs}

In Lemma \ref{L:opForm} we showed that the only differential operators $\hat f$ for which $\hat f \tr(A^{d(k-1)}) \neq 0$ are the operators whose associated digraphs are Eulerian.  The question remains: which $\hat f$ have an Eulerian associated digraph?  We answer this question with the following graph decoration. 

\begin{definition}
We define the \emph{$u$-rooted directed star of a $k$-uniform edge $e$} to be 
\[
S_e(u) = (e,\{uv: v \in e, u \neq v\}).
\]
A \emph{rooting} of a $k$-graph ${\cal H}$ is an ordering $R = (S_{e_{1}}(v_{1})
, S_{e_{2}}(v_{2}), \dots, S_{e_{m}}(v_m))$ such that $E({\cal H}) = \{e_1, \dots, e_m\}$ and $v_i \leq v_{i+1}$.  Given a rooting of $\cal H$ we define the \emph{rooted multi-digraph} of $R$ to be
\[
D_R = \bigcup_{i=1}^m S_{e_i}(v_i)
\]
where the union sums edge multiplicities.  We say that a rooting $R$ is an \emph{Euler rooting} if $D_R$ is Eulerian.  We denote the mutli-set of rooted digraphs of ${\cal H}$ as $S({\cal H})$.
\end{definition}
Note that two distinct rootings can yield the same rooted digraph.  We suppress the subscript $D_R$ and write $D$ when the context is clear.  We further refer to $D$ as a rooted digraph of ${\cal H}$ for convenience. 
\begin{definition}
Given a rooted digraph $D \in S({\cal H})$, we define the \emph{rooted operator of $D$ } to be
\[
\hat f_D = \prod_{uv \in E(D)} \frac{\partial}{\partial A_{u,v}}.
\]
Moreover, we denote
\[
\hat S({\cal H}) = \{\hat f_D : D \in R({\cal H})\}.
\]
In the case when $D$ is Eulerian we refer to $\hat f_D$ as an Euler operator.  
\end{definition}

The notation of $\hat f_D$ is consistent with our usage of $\hat f$ whence 
\[
\hat f_D \mid \prod_{i=1}^n \hat f_i^{d_i}
\]
where $d_i $ is equal to the number of times vertex $i$ is appears as a root of $D$.  If $\hat f$ is a rooted operator then it is understood that there exists a (not necessarily unique) rooting $R$ such that, with a slight abuse of notation, $\hat f = \hat f_{D_R}$.  We call such a rooting an {\em underlying rooting} of a differential operator.

\begin{lemma}
\label{L:EulerianRootings}
The associated digraph of an operator $\hat f$ is Eulerian if and only if $\hat f$ is an Euler operator.  
\end{lemma}

\begin{remark}
By Lemma \ref{L:EulerianRootings} the only operators which have non-zero contribution to $\Tr_d({\cal H})$ are Euler operators.  We denote $R({\cal H}) \subseteq S({\cal H})$ to be the multi-set of \emph{Euler rooted digraphs of ${\cal H}$}.  We further denote 
\[
\hat R({\cal H}) = \{\hat f_D : D \in R({\cal H})\}.
\]
\end{remark}

\begin{remark}
One can deduce Theorem 4.1 of \cite{Sha} from Lemma \ref{L:EulerianRootings} by a change of notation: our $\hat f_D$ is their $F$, our set of Eulerian associated digraphs arising from $\hat f_R$ is their ${\bf E}_{d,{m-1}}(n)$, and our $\mathfrak{E}(D)$ is their ${\bf W}(E)$.
\end{remark}

We now show that an Euler rooting is a rooting of a special type of hypergraph.
\begin{definition}
A \emph{Veblen hypergraph}\footnote{The nomenclature is a reference to Oswald Veblen (1880-1960) who proved an extension of Euler's theorem in 1912.  We present a brief note about Veblen's namesake theorem at the conclusion of this section.} is a $k$-uniform, $k$-valent multi-hypergraph. 
\end{definition}

\begin{lemma}
\label{L:RootingsVeblen}
An Euler rooting $R$ is a rooting of precisely one labeled Veblen hypergraph.
\end{lemma}
\begin{proof}
Suppose $S = (S_i)_{i=1}^m$ is a rooting of a connected $k$-graph ${\cal H}$.  Since $D_S$ is Eulerian we have for all $j \in V({\cal H})$  
\[
\deg^+(j) = (k-1)|\{i : v_i = j\}| = |\{i : v_i \neq j, j \in e_i\}| = \deg^-(j).
\]
  Fix a vertex $v \in V({\cal H})$.
We compute
\[
 \deg_{\cal H}(v) = \deg_D^+(v) + \deg_D^-(v) = k|\{i : v_i = v\}|.  
\]

Observe that $k \mid \deg_{\cal H}(v)$; it follows that ${\cal H}$ is Veblen by definition. Now suppose that ${\cal H}_0$ is a connected Veblen graph such that $S$ is an Euler rooting of ${\cal H}_0$.  As $S$ is a rooting of $H_0$, $E({\cal H}_0) = E({\cal H})$ and since both hypergraphs are connected $V({\cal H}_0) = V({\cal H})$. It follows that ${\cal H}$ is unique.
\end{proof}

We combine Lemmas \ref{L:EulerianRootings} and \ref{L:RootingsVeblen} into the following Lemma.

\begin{Lemma}
\label{P:Euler}
We have
\[
\hat f\tr(A^{d(k-1)}) \neq 0
\]
if and only if $\hat f = \hat f_D$ is a rooted operator.  Moreover, the underlying rooting of $\hat f$ is necessarily an Euler rooting of precisely one connected, labeled Veblen hypergraph.
\end{Lemma}

In the following section we use Lemma \ref{P:Euler} to express the codegree-$d$ coefficient of a $k$-graph as a function of Veblen hypergraphs.  Here we conclude with a note about Veblen's theorem. 

\begin{theorem}
\label{T:Veblen}
(Veblen's theorem \cite{Veb}) The set of edges of a finite graph can be written as a union of disjoint simple cycles if and only if every vertex has even degree.
\end{theorem}

Unfortunately, Veblen's theorem does not extend to higher uniformity: the set of edges of a finite $k$-graph ${\cal H}$ can not always be written as a union of disjoint simple $k$-regular $k$-graphs if and only if ${\cal H}$ is $k$-valent.  Consider the Veblen 3-graph ${\cal T}$ which consists of three bottomless tetrahedrons each sharing a common base.  To be precise, 
\[
{\cal T} = \left(\{a,b,c,1,2,3\}, \bigcup_{i=1}^3 \left( \binom{\{a,b,c,i\}}{3} \setminus \{a,b,c\} \right)\right).
\]
A drawing of ${\cal T}$ is given in Figure \ref{F:1}.
\begin{figure}[ht]
\begin{center}
\includegraphics[scale = .3]{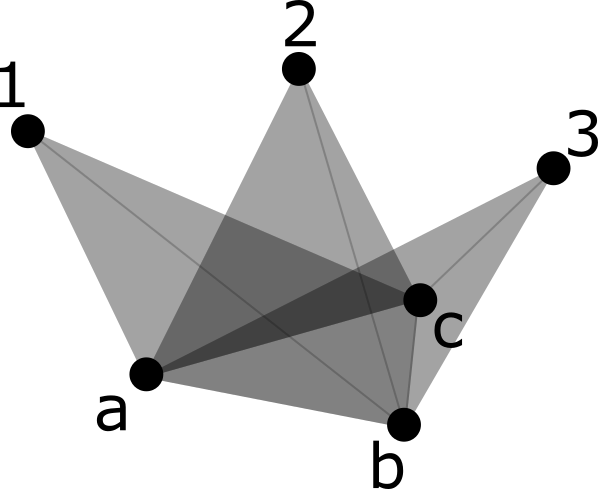} 
\end{center}
\caption{The Veblen 3-graph ${\cal T}$ where edges are drawn as triangular faces.}
\label{F:1}
\end{figure}
Since there are only three edges containing $i \in [3]$ it must be the case that any partition into Veblen graphs places each edge containing $i$ into the same class.  Observe that for each $i$ the vertices $a$, $b$, and $c$ each have degree $2$.  Therefore, the only $3$-valent edge partition is the trivial one.

\section{The associated coefficient of a Veblen hypergraph}

We now turn our attention to computing the codegree-$d$ coefficient of a $k$-graph ${\cal H}$ via Equation \ref{E:Morozov}.  From Lemma \ref{P:Euler} we know that the only operators which satisfy $\hat f \tr(A^{d(k-1)}) \neq 0$ are rooted operators.  Furthermore, as a differential operator of $\Tr_d({\cal A}_{\cal H})$ is of degree $d$, the underlying Euler rooting of $\hat f$ is a rooting of precisely one connected, labeled Veblen hypergraph with $d$ edges.  We equate $\Tr_d({\cal H})$ to a weighted sum over Euler rootings of connected Veblen graphs with $d$ edges which ``appear" in $\cal H$.  Consider the following generalization of the notion of subgraph.
  
\begin{definition}
For a labeled multi-hypergraph $H$, we call  the simple $k$-graph formed by removing duplicate edges of $H$ the \emph{flattening} of $H$ and denote it $\underline{H}$. We say that $H$ is an \emph{infragraph} of ${\cal H}$ if $\underline{H} \subseteq {\cal H}$.  Let ${\cal V}_d({\cal H})$ denote the set of isomorphism classes of connected, labeled Veblen infragraphs with $d$ edges of ${\cal H}$.
\end{definition}

\begin{definition}
The \emph{associated coefficient of a connected Veblen hypergraph} $H$ is 
\[
C_H = \sum_{D \in R(H)} \left(\frac{\tau_{D}}{\prod_{v \in V(D)} \deg^-(v)}\right).
\]
The \emph{associated coefficient} of a (possibly disconnected) Veblen hypergraph $H = \bigcup_{i=1}^m G_i$ is 
\[
C_H = \prod_{i=1}^m C_{G_i}.
\]
\end{definition}

  \begin{definition}
  For a $k$-graph ${\cal H}$ and a Veblen $k$-graph $H = \bigcup_{i=1}^m G_i$ we define 
\[
(\#H \subseteq {\cal H}) =\frac{1}{|\Aut(H)|}\prod_{i=1}^m|\Aut(\underline{G_i})||\{S \subseteq {\cal H} : S \cong \underline{G_i}\}|.
\]
\end{definition}
In the case when $H$ is connected this simplifies to
 \[
 (\#H \subseteq {\cal H}) = \frac{|\Aut(\underline{H})|}{|\Aut(H)|}|\{S \subseteq {\cal H} : S \cong \underline{H}\}| = |\Aut(\underline{H})/\Aut(H)|\cdot |\{S \subseteq {\cal H} : S \cong \underline{H}\}|.
 \]
 
Note that for $H = \bigcup_{i=1}^m G_i$, $(H \subseteq {\cal H})$ is not multiplicative over the components of $H$ as
\[
\prod_{i=1}^m(\#G_i \subseteq {\cal H}) = \frac{(\#H \subseteq {\cal H})|\Aut(H)|}{\prod_{i=1}^m|\Aut(G_i)|}.
\]
However, we have the following identity.
\begin{lemma}
\label{L:mu}
Let $H = \bigcup_{i=1}^m G_i$ be a Veblen $k$-graph.  If $\mu_H$ denotes the number of linear orderings of the components of $H$ (where two components are indistinguishable if they are isomorphic) then  
\[
 (\#H \subseteq {\cal H}) = \frac{\mu_H}{m!}\prod_{i=1}^m(\#G_i \subseteq {\cal H}).
\]
\end{lemma}
\begin{proof}
Suppose there are $t$ isomorphism classes of components of $H$ with representatives $H_1, H_2, \dots H_t$.  Denote the number of components of $H$ which are isomorphic to $H_i$ as $\mu_i$. Fix an ordering of the components which are isomorphic to $H_i$, $\{G_{1,i},G_{2,i}, \dots, G_{\mu_i,i}\}$.  The number of distinct linear orderings of the components of $H$ where $G_i$ and $G_j$ are indistinguishable when $G_{r,i} \cong G_{s,i}$ is
\[
\mu _{H} = \binom{m}{\mu_1, \mu_2, \dots,  \mu_t}
\]
so that 
\[
\frac{m!}{\mu_H} = \prod_{i=1}^t \mu_i!.
\]
Note that for $a \in \Aut(H)$, there exists $\sigma \in \mathfrak{S}_{\mu_i}$ such that $a(G_{j,i}) = G_{\sigma(j),i}$.  In this way, $\Aut(H)$ induces a permutation on the isomorphism classes of the components of $H$ (note that this map is well-defined since the components are labeled).  Let $\psi: \Aut(H) \to \mathfrak{S}_{\mu_1} \times \mathfrak{S}_{\mu_2}\times \dots \times \mathfrak{S}_{\mu_t}$ be such a map.  Notice $\ker(\psi)$ is the group of automorphisms of $H$ which maps each component to itself.  Appealing to the First Isomorphism Theorem we have
\[
\frac{|\Aut(H)|}{\prod_{i=1}^m|\Aut(G_i)|} = \prod_{i=1}^t \mu_i!.
\]
The desired equality follows by substitution. 
\end{proof}

\begin{remark}
The equation in Lemma \ref{L:mu} implies that $(\#H \subseteq {\cal H})$ is multiplicative over its components if and only if the components of $H$ are pairwise non-isomorphic.
\end{remark}

Let 
\[
A(d,n) = \left\{(d_1, \dots, d_n) : \sum d_i = d, d_i \geq 0 \right\}
\]
be the set of \emph{arrangements of $d$ into $n$ non-negative parts} and further let $A^+(d,n) \subset A(d,n)$ be the set of \emph{arrangements of $d$ into $n$ positive parts}.  For a $k$-graph ${\cal H}$ and $a \in A^+(d,|E|)$ let $R^a({\cal H})$ be the set of Euler rootings of all labeled, connected Veblen infragraphs of ${\cal H}$ which have the property that vertex $v_i$ is the root of exactly $d_i$ edges.  (N.B. We take $a \in A^+(d,|E|)$ as it is necessary that $d_i > 0$ for $D \in R^a({\cal H})$ to be Eulerian.)

\begin{remark}
Let ${\cal V}^*_d({\cal H})$ denote the set of (possibly disconnected) Veblen infragraphs of ${\cal H}$ with $d$ edges up to isomorphism.  Further let ${\cal V}_d({\cal H})\subseteq {\cal V}^*_d({\cal H})$ denote the set of connected Veblen infragraphs of ${\cal H}$ with $d$ edges up to isomorphism. 
\end{remark}

We now present a formula for $\Tr_d({\cal H})$ as a weighted sum over its Veblen infragraphs.  
\begin{lemma}
\label{L:Trace}
  For a $k$-graph ${\cal H}$
\[
\Tr_d({\cal H}) = d(k-1)^n \sum_{H \in {\cal V}_d({\cal H})} C_H(\#H \subseteq {\cal H}).
\]
\end{lemma}
\begin{proof}
For convenience let $|E| = |E({\cal H})|$.  We equate
\begin{align*}
    \sum_{H \in {\cal V}_d({\cal H})} C_H(\#H \subseteq {\cal H})&=\sum_{H \in {\cal V}_d({\cal H})}\left(\left(\sum_{D \in R(H)}\frac{\tau_D}{\prod_{v \in V(D))} \deg^-(v) }\right) (\# H \subseteq {\cal H})\right) \\
    &=\sum_{a \in A^+(d,|E|)} \left(\sum_{ D \in R^a({\cal H})}\frac{\tau_{D}}{\prod_{v \in V(D)} \deg^-(v) }\right).
\end{align*}
Recall 
\[
\Tr_d({\cal H}) = (k-1)^{n-1} \sum_{d_1 + \dots + d_n = d} \left(\prod_{i=1}^n \frac{\hat f_i^{d_i}}{(d_i(k-1))!}\tr(A^{d(k-1)})\right).
\]
  Applying Lemma \ref{P:Euler} we have
\[
\Tr_d({\cal H}) = (k-1)^{n-1} \sum_{a \in A^+(d,|E|)} \left(\sum_{D \in R^a({\cal H})}\frac{\hat f_{D}\tr(A^{d(k-1)})}{\prod_{i=1}^n(d_i(k-1))!}\right).
\]
  By Corollary \ref{C:opForm} ,
\[
\hat f_{D} \tr(A^{d(k-1)}) = d(k-1)\tau_{D} \prod_{v \in V(D_{R})} (\deg^-(v) -1)!
\]
When $D \in R^a({\cal H})$ with $a = (d_1,\ldots,d_n)$, we have $\deg^-(v_i) = d_i(k-1)$.  By substitution we have 
\begin{align*}
\Tr_d({\cal H}) &= d(k-1)^{n} \sum_{a \in A^+(d,|E|)} \left(\sum_{ D \in R^a({\cal H})}\frac{\tau_{D}}{\prod_{v \in V(D)} \deg^-(v) }\right) \\
 &= d(k-1)^n \sum_{H \in {\cal V}_d({\cal H})} C_H(\#H \subseteq {\cal H}).
 \end{align*}
\end{proof}

We are now prove a generalization of the Harary-Sachs formula for $k$-graphs. 

\begin{theorem}
\label{T:codegreeFormula}
For a simple $k$-graph $\cal H$, 
\[
\phi_d({\cal H}) = \sum_{H \in {{\cal V}^*_d}({\cal H})}(-(k-1)^n)^{c(H)}C_H(\# H \subseteq {\cal H}).
\]
\end{theorem}
\begin{proof}
Fix a $k$-graph ${\cal H}$ and $d \geq 1$. From \cite{Coo} we have by Equation \ref{E:Morozov}
\[
\phi_d({\cal H}) = P_d\left(\frac{-\Tr_1({\cal H})}{1}, \frac{-\Tr_2({\cal H})}{2},\dots, \frac{-\Tr_d({\cal H})}{d}\right)
\]
 where 
\[
P_d(t_1, t_2, \dots, t_d) = \sum_{m=1}^d \sum_{d_1 + \cdots+d_m = d} \frac{t_{d_1}t_{d_2}\cdots t_{d_m}}{m!}.
\]
 
Fix an $1 \leq m \leq d$ and an arrangement $a = (d_1, d_2, \dots, d_m) \in A^+(d,m)$.  By Lemma \ref{L:Trace}, 
\[
\frac{-\Tr_{d_i}({\cal H})}{d_i} = -(k-1)^n \sum_{G \in {\cal V}_{d_i}({\cal H})} C_{G} (\#G \subseteq {\cal H}).
\]
 Let ${\cal V}^a({\cal H})$ denote the set of $m$-tuples of connected unlabeled Veblen infragraphs of ${\cal H}$ whose $i$-th coordinate has $d_i$ edges for $i \in [m]$, such that $\sum_i d_i = d$.
 We have
\begin{align*}
    \prod_{i=1}^m \frac{-\Tr_{d_i}({\cal H})}{d_i} &= (-1)^m(k-1)^{mn}\prod_{i=1}^m \sum_{G \in {\cal V}_{d_i}({\cal H})}C_{G}(\#G \subseteq {\cal H}) \\
    & = (-(k-1)^n)^{m}\sum_{\substack{H = G_1 \cup \cdots \cup G_m \\(G_1, G_2, \dots, G_m) \in {\cal V}^a({\cal H})}}C_{H} \prod_{i=1}^m (\#G_i \subseteq {\cal H})
\end{align*}

For $m \in \mathbb{N}$, let ${\cal V}_d^m({\cal H})$ be the set of unlabeled Veblen infragraphs of ${\cal H}$ with $d$ edges and $m$ components.  Appealing to Lemma \ref{L:mu} we may write
\begin{align*}
\phi_d&=P_d\left(-\frac{\Tr_1({\cal H})}{1},-\frac{\Tr_2({\cal H})}{2}, \dots, -\frac{\Tr_d({\cal H})}{d}\right) \\
&= \sum_{m=1}^d \sum_{d_1+\dots + d_m = d} \frac{1}{m!} \prod_{i=1}^m \frac{-\Tr_{d_i}({\cal H})}{d_i} \\
&= \sum_{m=1}^d \left(\sum_{a \in A^+(m,d)} (-(k-1)^n)^{m}\sum_{H \in {\cal V}^a({\cal H}) }C_H\frac{\left(\prod_{i=1}^m (\#G_i \subseteq {\cal H})\right)}{m!} \right) \\
&= \sum_{m=1}^d (-(k-1)^n)^{m}\sum_{H \in {\cal V}^m_d({\cal H})}C_H\left(\frac{\mu_H}{m!}\prod_{i=1}^m(\# G_i \subseteq {\cal H})\right) \\
&=  \sum_{H \in {\cal V}^*_d({\cal H})}(-(k-1)^n)^{c(H)}C_H(\# H \subseteq {\cal H}).
\end{align*}
\end{proof}

\bibliography{main.bib}{}
\bibliographystyle{plain}

\end{document}